\documentclass[twoside,11pt,reqno]{amsart}
\usepackage{amsmath,amsthm,amssymb,amstext,amsfonts,amscd}
\usepackage{graphicx}
\usepackage{multirow}

\newtheorem{theorem}{Theorem}

\newtheorem{definition}{Definition}

\newtheorem{remark}{Remark}

\numberwithin{equation}{section}
\input{tcilatex}

\begin{document}
\title[ Some growth properties of entire and meromorphic functions.....]{%
{\small \ }$\left( p,q\right) $- th relative order and $\left( p,q\right) $-
th relative type based some growth analysis of entire and meromorphic
functions}
\author[Tanmay Biswas]{Tanmay Biswas}
\address{T. Biswas : Rajbari, Rabindrapalli, R. N. Tagore Road, P.O.-
Krishnagar, Dist-Nadia, PIN-\ 741101, West Bengal, India}
\email{tanmaybiswas\_math@rediffmail.com}

\begin{abstract}
{\small In this paper we wish to prove some results related to the growth
rates of entire and meromorphic functions on the basis of relative }$\left(
p,q\right) ${\small \ th order and relative }$\left( p,q\right) ${\small \
th type of a meromorphic function with respect to an entire function for any
two positive integers }$p${\small \ and }$q${\small .}
\end{abstract}

\maketitle

\section{\textbf{Introduction, Definitions and Notations}}

\qquad Let $f$ be an entire function defined in the open complex plane $%
\mathbb{C}
.$ The\ maximum modulus function $M_{f}\left( r\right) $ corresponding to $f$%
\ is defined on $\left\vert z\right\vert =r$ as $M_{f}\left( r\right) =%
\QTATOP{\max }{\left\vert z\right\vert =r}\left\vert f\left( z\right)
\right\vert $. If an entire function $f$ is non-constant then $M_{f}\left(
r\right) $ is strictly increasing and continuous and its inverse $%
M_{f}{}^{-1}:\left( \left\vert f\left( 0\right) \right\vert ,\infty \right)
\rightarrow \left( 0,\infty \right) $ exists and is such that $\underset{%
s\rightarrow \infty }{\lim }M_{f}^{-1}\left( s\right) =\infty .$ When $f$ is
meromorphic, one may introduce another function $T_{f}\left( r\right) $
known as Nevanlinna's characteristic function of $f,$ playing the same role
as $M_{f}\left( r\right) .$

\qquad The integrated counting function $N_{f}\left( r,a\right) \left( 
\overline{N}_{f}\left( r,a\right) \right) $ of $a$-points (distinct $a$%
-points) of $f$ is defined as%
\begin{equation*}
N_{f}\left( r,a\right) =\overset{r}{\underset{0}{\int }}\frac{n_{f}\left(
t,a\right) -n_{f}\left( 0,a\right) }{t}dt+n_{f}\left( 0,a\right) \log r
\end{equation*}%
\begin{equation*}
\left( \overline{N}_{f}\left( r,a\right) =\overset{r}{\underset{0}{\int }}%
\frac{\overline{n}_{f}\left( t,a\right) -\overline{n}_{f}\left( r,a\right) }{%
t}dt+\overline{n}_{f}\left( 0,a\right) \log r\right) ~,
\end{equation*}%
where we denote by $n_{f}\left( t,a\right) \left( \overline{n}_{f}\left(
t,a\right) \right) $ the number of $a$-points (distinct $a$-points) of $f$
in $\left\vert z\right\vert \leq t$ and an $\infty $ -point is a pole of $f$
. In many occasions $N_{f}\left( r,\infty \right) $ and $\overline{N}%
_{f}\left( r,\infty \right) $ are denoted by $N_{f}\left( r\right) $ and $%
\overline{N}_{f}\left( r\right) $ respectively. The function $N_{f}\left(
r,a\right) $ is called the enumerative function. On the other hand, the
function $m_{f}\left( r\right) \equiv m_{f}\left( r,\infty \right) $ known
as the proximity function is defined as%
\begin{align*}
m_{f}\left( r\right) & =\frac{1}{2\pi }\overset{2\pi }{\underset{0}{\int }}%
\log ^{+}\left\vert f\left( re^{i\theta }\right) \right\vert d\theta , \\
\text{where }\log ^{+}x& =\max \left( \log x,0\right) \text{ for all }%
x\geqslant 0
\end{align*}%
and an $\infty $ -point is a pole of $f$ .

\qquad Analogously, $m_{\frac{1}{f-a}}\left( r\right) \equiv m_{f}\left(
r,a\right) $ is defined when $a$ is not an $\infty $-point of $f.$

\qquad Thus the Nevanlinna's characteristic function $T_{f}\left( r\right) $
corresponding to $f$ is defined as%
\begin{equation*}
T_{f}\left( r\right) =N_{f}\left( r\right) +m_{f}\left( r\right) ~.
\end{equation*}

\qquad When $f$ is entire, $T_{f}\left( r\right) $ coincides with $%
m_{f}\left( r\right) $ as $N_{f}\left( r\right) =0.$

\qquad Moreover, if $f$ is non-constant entire then $T_{f}\left( r\right) $
is strictly increasing and continuous functions of $r$. Also its inverse $%
T_{f}^{-1}:\left( T_{f}\left( 0\right) ,\infty \right) \rightarrow \left(
0,\infty \right) $ exist and is such that $\underset{s\rightarrow \infty }{%
\lim }T_{f}^{-1}\left( s\right) =\infty $. Also the ratio $\frac{T_{f}\left(
r\right) }{T_{g}\left( r\right) }$ as $r\rightarrow \infty $ is called the
growth of $f$ with respect to $g$ in terms of the Nevanlinna's
Characteristic functions of the meromorphic functions $f$ and $g$.

\qquad Now we state the following notation which will be needed in the
sequel:

\begin{align*}
\log ^{[k]}x& =\log \left( \log ^{[k-1]}x\right) \text{ for }k=1,2,3,\cdot
\cdot \cdot \text{ and} \\
\log ^{[0]}x& =x;
\end{align*}%
and%
\begin{align*}
\exp ^{[k]}x& =\exp \left( \exp ^{[k-1]}x\right) \text{ for }k=1,2,3,\cdot
\cdot \cdot \text{ and} \\
\exp ^{[0]}x& =x.
\end{align*}

\qquad Taking this into account, Juneja, Kapoor and Bajpai \cite{5} defined
the $(p,q)$-th order and $(p,q)$-th lower order of an entire function $f$
respectively as follows:

\begin{equation*}
\rho _{f}\left( p,q\right) =\underset{r\rightarrow \infty }{\lim \sup }\frac{%
\log ^{[p]}M_{f}(r)}{\log ^{\left[ q\right] }r}\text{ and }\lambda
_{f}\left( p,q\right) =\underset{r\rightarrow \infty }{\lim \inf }\frac{\log
^{[p]}M_{f}(r)}{\log ^{\left[ q\right] }r},
\end{equation*}%
where $p,q$ are any two positive integers with $p\geq q$. When $f$ is
meromorphic one can easily verify that%
\begin{equation*}
\rho _{f}\left( p,q\right) =\underset{r\rightarrow \infty }{\lim \sup }\frac{%
\log ^{[p-1]}T_{f}(r)}{\log ^{\left[ q\right] }r}\text{ and }\lambda
_{f}\left( p,q\right) =\underset{r\rightarrow \infty }{\lim \inf }\frac{\log
^{[p-1]}T_{f}(r)}{\log ^{\left[ q\right] }r},
\end{equation*}%
where $p,q$ are any two positive integers with $p\geq q$.If $p=l$ and $q=1$
then we write $\rho _{f}\left( l,1\right) =\rho _{f}^{\left[ l\right] }$ and 
$\lambda _{f}\left( l,1\right) =\lambda _{f}^{\left[ l\right] }$ where $\rho
_{f}^{\left[ l\right] }$ and $\lambda _{f}^{\left[ l\right] }$ are
respectively known as generalized order and generalized lower order of $f$.
Also for $p=2$ and $q=1$ we respectively denote $\rho _{f}\left( 2,1\right) $
and $\lambda _{f}\left( 2,1\right) $ by $\rho _{f}$ and $\lambda _{f}.$
where $\rho _{f}$ and $\lambda _{f}$ are the classical growth indicator
known as order and lower order of $f$.

\qquad In this connection we just recall the following definition :

\begin{definition}
\label{d1}\cite{5} An entire function $f$ is said to have index-pair $\left(
p,q\right) $, $p\geq q\geq 1$ if $b<\rho _{f}\left( p,q\right) <\infty $ and 
$\rho _{f}\left( p-1,q-1\right) $ is not a nonzero finite number, where $b=1$
if $p=q$ and $b=0$ if $p>q.$ Moreover if $0<\rho _{f}\left( p,q\right)
<\infty ,$ then%
\begin{equation*}
\rho _{f}\left( p-n,q\right) =\infty ~\text{for }n<p,~\rho _{f}\left(
p,q-n\right) =0~\text{for }n<q\text{ and}
\end{equation*}%
\begin{equation*}
\rho _{f}\left( p+n,q+n\right) =1~\text{for }n=1,2,....~.
\end{equation*}%
Similarly for $0<\lambda _{f}\left( p,q\right) <\infty ,$ one can easily
verify that%
\begin{equation*}
\lambda _{f}\left( p-n,q\right) =\infty ~\text{for }n<p,~\lambda _{f}\left(
p,q-n\right) =0~\text{for }n<q\text{ and}
\end{equation*}%
\begin{equation*}
\lambda _{f}\left( p+n,q+n\right) =1~\text{for }n=1,2,....~.
\end{equation*}%
An entire function for which $(p,q)$-th order and $(p,q)$-th lower order are
the same is said to be of regular $\left( p,q\right) $-growth. Functions
which are not of regular $\left( p,q\right) $-growth are said to be of
irregular $\left( p,q\right) $-growth.
\end{definition}

\qquad Analogously one can easily verify that the Definition \ref{d1} of
index-pair can also be applicable for a meromorphic function $f$.

\qquad In order to compare the growth of entire functions having the same $%
(p,q)$-th order, Juneja, Kapoor and Bajpai \cite{Juneja2} also introduced
the concepts of $(p,q)$-th type and\ $(p,q)$-th lower type in the following
manner :

\begin{definition}
\label{d2}\cite{Juneja2} The $\left( p,q\right) $-th type and the $\left(
p,q\right) $-th lower type of entire function $f$ having finite positive $%
\left( p,q\right) $-th order $\rho _{f}\left( p,q\right) $ $(b$ $<$ $\rho
_{f}\left( p,q\right) $ $<$ $\infty )$ ($p,q$ are any two positive integers, 
$b=1$ if $p=q$ and $b=0$ for $p>q)$ are defined as :%
\begin{equation*}
\sigma _{f}\left( p,q\right) =\underset{r\rightarrow \infty }{\lim \sup }%
\frac{\log ^{\left[ p-1\right] }M_{f}\left( r\right) }{\left( \log ^{\left[
q-1\right] }r\right) ^{\rho _{f}\left( p,q\right) }}\text{\ and \ }\overline{%
\sigma }_{f}\left( p,q\right) =\underset{r\rightarrow \infty }{\lim \inf }%
\frac{\log ^{\left[ p-1\right] }M_{f}\left( r\right) }{\left( \log ^{\left[
q-1\right] }r\right) ^{\rho _{f}\left( p,q\right) }}\text{,}
\end{equation*}%
\begin{equation*}
~\ \ \ \ \ \ \ \ \ \ \ \ \ \ \ \ \ \ \ \ \ \ \ \ \ \ \ \ \ \ \ \ \ \ \ \ \ \
\ \ \ \ \ \ \ \ \ \ \ \ \ \ \ \ \ \ \ \ \ \ \ \ \ \ \ \ \ \text{ }0\leq 
\overline{\sigma }_{f}\left( p,q\right) \leq \sigma _{f}\left( p,q\right)
\leq \infty \text{~.}
\end{equation*}%
When $f$ is meromorphic one can easily verify that%
\begin{equation*}
\sigma _{f}\left( p,q\right) =\underset{r\rightarrow \infty }{\lim \sup }%
\frac{\log ^{\left[ p-2\right] }T_{f}\left( r\right) }{\left( \log ^{\left[
q-1\right] }r\right) ^{\rho _{f}\left( p,q\right) }}\text{\ and \ }\overline{%
\sigma }_{f}\left( p,q\right) =\underset{r\rightarrow \infty }{\lim \inf }%
\frac{\log ^{\left[ p-2\right] }T_{f}\left( r\right) }{\left( \log ^{\left[
q-1\right] }r\right) ^{\rho _{f}\left( p,q\right) }}\text{,}
\end{equation*}%
\begin{equation*}
~\ \ \ \ \ \ \ \ \ \ \ \ \ \ \ \ \ \ \ \ \ \ \ \ \ \ \ \ \ \ \ \ \ \ \ \ \ \
\ \ \ \ \ \ \ \ \ \ \ \ \ \ \ \ \ \ \ \ \ \ \ \ \ \ \ \ \ \text{ }0\leq 
\overline{\sigma }_{f}\left( p,q\right) \leq \sigma _{f}\left( p,q\right)
\leq \infty \text{~.}
\end{equation*}
\end{definition}

\qquad Likewise, to compare the growth of entire functions having the same $%
(p,q)$-th lower order, one can also introduced the concepts of $(p,q)$-th
weak type in the following manner :

\begin{definition}
\label{d3} The $\left( p,q\right) $ th weak type of entire function $f$
having finite positive $\left( p,q\right) $ th tower order $\lambda
_{f}\left( p,q\right) $ $(b$ $<$ $\lambda _{f}\left( p,q\right) $ $<$ $%
\infty )$ is defined as :%
\begin{equation*}
\tau _{f}\left( p,q\right) =\underset{r\rightarrow \infty }{\lim \inf }\frac{%
\log ^{\left[ p-1\right] }M_{f}\left( r\right) }{\left( \log ^{\left[ q-1%
\right] }r\right) ^{\lambda _{f}\left( p,q\right) }}\text{~}
\end{equation*}%
where $p,q$ are any two positive integers, $b=1$ if $p=q$ and $b=0$ for $%
p>q~.$

Similarly one may define the growth indicator $\overline{\tau }_{f}\left(
p,q\right) $ of an entire function $f$ in the following way :%
\begin{equation*}
\overline{\tau }_{f}\left( p,q\right) =\underset{r\rightarrow \infty }{\lim
\sup }\frac{\log ^{\left[ p-1\right] }M_{f}\left( r\right) }{\left( \log ^{%
\left[ q-1\right] }r\right) ^{\lambda _{f}\left( p,q\right) }}\text{ },\text{
}b<\lambda _{f}\left( p,q\right) <\infty
\end{equation*}%
where $p,q$ are any two positive integers, $b=1$ if $p=q$ and $b=0$ for $%
p>q~.$\newline
When $f$ is meromorphic one can easily verify that%
\begin{eqnarray*}
\tau _{f}\left( p,q\right) &=&\underset{r\rightarrow \infty }{\lim \inf }%
\frac{\log ^{\left[ p-2\right] }M_{f}\left( r\right) }{\left( \log ^{\left[
q-1\right] }r\right) ^{\lambda _{f}\left( p,q\right) }}\text{\ and } \\
\text{\ }\overline{\tau }_{f}\left( p,q\right) &=&\underset{r\rightarrow
\infty }{\lim \sup }\frac{\log ^{\left[ p-2\right] }M_{f}\left( r\right) }{%
\left( \log ^{\left[ q-1\right] }r\right) ^{\lambda _{f}\left( p,q\right) }}%
\text{ },\text{ }b<\lambda _{f}\left( p,q\right) <\infty \text{,}
\end{eqnarray*}

It is obvious that $0\leq \tau _{f}\left( p,q\right) \leq \overline{\tau }%
_{f}\left( p,q\right) \leq \infty $~$.$
\end{definition}

\qquad Given a non-constant entire function $f$ defined in the open complex
plane $%
\mathbb{C}
$ its maximum modulus function and Nevanlinna's characteristic function are
strictly increasing and continuous. Hence there exists its inverse functions 
$M_{f}^{-1}(r):\left( \left\vert f\left( 0\right) \right\vert ,\infty
\right) \rightarrow \left( 0,\infty \right) $ with $\underset{s\rightarrow
\infty }{\lim }M_{f}^{-1}\left( s\right) =\infty $ and $T_{f}^{-1}(r):\left(
\left\vert f\left( 0\right) \right\vert ,\infty \right) \rightarrow \left(
0,\infty \right) $ with $\underset{s\rightarrow \infty }{\lim }%
T_{f}^{-1}\left( s\right) =\infty .$

\qquad In this connection, Bernal \cite{1} introduced the definition of
relative order of an entire function $f$ with respect to another entire
function $g$, denoted by $\rho _{g}\left( f\right) $ as follows:%
\begin{eqnarray*}
\rho _{g}\left( f\right) &=&\inf \left\{ \mu >0:M_{f}\left( r\right)
<M_{g}\left( r^{\mu }\right) \text{ for all }r>r_{0}\left( \mu \right)
>0.\right\} \\
&=&\underset{r\rightarrow \infty }{\lim \sup }\frac{\log
M_{g}^{-1}M_{f}\left( r\right) }{\log r}~.
\end{eqnarray*}

\qquad The definition coincides with the classical one \cite{11} if $g\left(
z\right) =\exp z.$ Similarly one can define the relative lower order of an
entire function $f$ with respect to another entire function $g$ denoted by $%
\lambda _{g}\left( f\right) $ as follows :%
\begin{equation*}
\lambda _{g}\left( f\right) =\underset{r\rightarrow \infty }{\lim \inf }%
\frac{\log M_{g}^{-1}M_{f}\left( r\right) }{\log r}~.
\end{equation*}

\qquad Extending this notion, Ruiz et. al. \cite{xx} introduced the
definition of $\left( p,q\right) $ th relative order of a entire function
with respect to an entire function in the light of index pair which is as
follows :

\begin{definition}
\label{d4} \cite{xx}\textbf{\ }Let $f$ and $g$ be any two entire functions
with index-pairs $\left( m,q\right) $ and $\left( m,p\right) $ respectively
where $p,q,m$ are positive integers such that $m\geq \max (p,q).$ Then the
relative $\left( p,q\right) $-th order of $f$ with respect to $g$ is defined
as 
\begin{equation*}
\rho _{g}^{\left( p,q\right) }\left( f\right) =\text{ }\underset{%
r\rightarrow \infty }{\lim \sup }\frac{\log ^{\left[ p\right]
}M_{g}^{-1}M_{f}\left( r\right) }{\log ^{\left[ q\right] }r}.
\end{equation*}%
Analogously, the relative $\left( p,q\right) $-th lower order of $f$ with
respect to $g$ is defined by: 
\begin{equation*}
\lambda _{g}^{\left( p,q\right) }\left( f\right) =\underset{r\rightarrow
\infty }{\text{ }\lim \inf }\frac{\log ^{\left[ p\right] }M_{g}^{-1}M_{f}%
\left( r\right) }{\log ^{\left[ q\right] }r}.
\end{equation*}
\end{definition}

\qquad In order to refine the above growth scale, now we intend to introduce
the definitions of an another growth indicators, such as relative $\left(
p,q\right) $\ -th type and relative $\left( p,q\right) $\ -th lower type%
\emph{\ }of entire function with respect to another entire function in the
light of their index-pair which are as follows:

\begin{definition}
\label{d5} Let $f$ and $g$ be any two entire functions with index-pairs $%
\left( m_{1},q\right) $ and $\left( m_{2},p\right) $ respectively where $%
m_{1}=m_{2}=m$ and $p,q,m$ are all positive integers such that $m\geq \max
\left\{ p,q\right\} .$ The relative $\left( p,q\right) $\ -th type and
relative $\left( p,q\right) $\ -th lower type of entire function $f$ with
respect to the entire function $g$ having finite positive relative $\left(
p,q\right) $ th order $\rho _{g}^{\left( p,q\right) }\left( f\right) $ $%
\left( 0<\rho _{g}^{\left( p,q\right) }\left( f\right) <\infty \right) $ are
defined as :%
\begin{equation*}
\sigma _{g}^{\left( p,q\right) }\left( f\right) =\underset{r\rightarrow
\infty }{\lim \sup }\frac{\log ^{\left[ p-1\right] }M_{g}^{-1}M_{f}\left(
r\right) }{\left( \log ^{\left[ q-1\right] }r\right) ^{\rho _{g}^{\left(
p,q\right) }\left( f\right) }}\text{ \ \ and \ \ }\overline{\sigma }%
_{g}^{\left( p,q\right) }\left( f\right) =\underset{r\rightarrow \infty }{%
\lim \inf }\frac{\log ^{\left[ p-1\right] }M_{g}^{-1}M_{f}\left( r\right) }{%
\left( \log ^{\left[ q-1\right] }r\right) ^{\rho _{g}^{\left( p,q\right)
}\left( f\right) }}
\end{equation*}
\end{definition}

\qquad Analogously, to determine the relative growth of two entire functions
having same non zero finite relative $\left( p,q\right) $\ -th lower order
with respect to another entire function, one can introduced the definition
of relative $\left( p,q\right) $\ -th weak type of an entire function $f$
with respect to another entire function $g$ of finite positive relative $%
\left( p,q\right) $\ -th lower order $\lambda _{g}^{\left( p,q\right)
}\left( f\right) $ in the following way:

\begin{definition}
\label{d6} Let $f$ and $g$ be any two entire functions having finite
positive relative $\left( p,q\right) $ th lower order $\lambda _{g}^{\left(
p,q\right) }\left( f\right) $ $\left( 0<\lambda _{g}^{\left( p,q\right)
}\left( f\right) <\infty \right) $ where $p$\ and $q$ are any two positive
integers. Then the \emph{relative }$\left( p,q\right) $\emph{\ th weak type}
of entire function $f$ with respect to the entire function $g$ is defined as
:%
\begin{equation*}
\tau _{g}^{\left( p,q\right) }\left( f\right) =\underset{r\rightarrow \infty 
}{\lim \inf }\frac{\log ^{\left[ p-1\right] }M_{g}^{-1}M_{f}\left( r\right) 
}{\left( \log ^{\left[ q-1\right] }r\right) ^{\lambda _{g}^{\left(
p,q\right) }\left( f\right) }}~.
\end{equation*}%
Similarly one can define another growth indicator $\overline{\tau }%
_{g}^{\left( p,q\right) }\left( f\right) $ in the following way:%
\begin{equation*}
\overline{\tau }_{g}^{\left( p,q\right) }\left( f\right) =\underset{%
r\rightarrow \infty }{\lim \sup }\frac{\log ^{\left[ p-1\right]
}M_{g}^{-1}M_{f}\left( r\right) }{\left( \log ^{\left[ q-1\right] }r\right)
^{\lambda _{g}^{\left( p,q\right) }\left( f\right) }}~.
\end{equation*}
\end{definition}

\qquad In the case of relative order, it therefore seems reasonable to
define suitably the relative $\left( p,q\right) $ th order of meromorphic
functions. Debnath et. al. \cite{yy} also introduced such definition in the
light of index pair in the following manner:

\begin{definition}
\label{d7}\cite{yy} Let $f$ be any meromorphic function and $g$ be any
entire function with index-pairs $\left( m_{1},q\right) $ and $\left(
m_{2},p\right) $ respectively where $m_{1}=m_{2}=m$ and $p,q,m$ are all
positive integers such that $m\geq p$ and $m\geq q.$ Then the relative $%
\left( p,q\right) $ th order of $f$ with respect to $g$ is defined as%
\begin{equation*}
\rho _{g}^{\left( p,q\right) }\left( f\right) =\underset{r\rightarrow \infty 
}{\lim \sup }\frac{\log ^{\left[ p\right] }T_{g}^{-1}T_{f}\left( r\right) }{%
\log ^{\left[ q\right] }r}.
\end{equation*}

Similarly, one can define the relative $\left( p,q\right) $ th lower order
of a meromorphic function $f$ with respect to an entire function $g$ denoted
by $\lambda _{g}^{\left( p,q\right) }\left( f\right) $ where $p$ and $q$ are
any two positive integers in the following way:%
\begin{equation*}
\lambda _{g}^{\left( p,q\right) }\left( f\right) =\underset{r\rightarrow
\infty }{\lim \inf }\frac{\log ^{\left[ p\right] }T_{g}^{-1}T_{f}\left(
r\right) }{\log ^{\left[ q\right] }r}.
\end{equation*}
\end{definition}

\qquad Now we state the following two definitions relating to the
meromorphic function which are similar to Definition \ref{d6} and Definition %
\ref{d7} respectively.

\begin{definition}
\label{d8} Let $f$ be a meromorphic function and $g$ be an entire function
with index-pairs $\left( m_{1},q\right) $ and $\left( m_{2},p\right) $
respectively where $m_{1}=m_{2}=m$ and $p,q,m$ are all positive integers
such that $m\geq \max \left\{ p,q\right\} .$ The relative $\left( p,q\right) 
$\ -th type and relative $\left( p,q\right) $\ -th lower type of meromorphic
function $f$ with respect to the entire function $g$ having finite positive
relative $\left( p,q\right) $ th order $\rho _{g}^{\left( p,q\right) }\left(
f\right) $ $\left( 0<\rho _{g}^{\left( p,q\right) }\left( f\right) <\infty
\right) $ are defined as :%
\begin{equation*}
\sigma _{g}^{\left( p,q\right) }\left( f\right) =\underset{r\rightarrow
\infty }{\lim \sup }\frac{\log ^{\left[ p-1\right] }T_{g}^{-1}T_{f}\left(
r\right) }{\left( \log ^{\left[ q-1\right] }r\right) ^{\rho _{g}^{\left(
p,q\right) }\left( f\right) }}\text{ \ \ and \ \ }\overline{\sigma }%
_{g}^{\left( p,q\right) }\left( f\right) =\underset{r\rightarrow \infty }{%
\lim \inf }\frac{\log ^{\left[ p-1\right] }T_{g}^{-1}T_{f}\left( r\right) }{%
\left( \log ^{\left[ q-1\right] }r\right) ^{\rho _{g}^{\left( p,q\right)
}\left( f\right) }}
\end{equation*}
\end{definition}

\begin{definition}
\label{d9} Let $f$ be a meromorphic function and $g$ be an entire function
having finite positive relative $\left( p,q\right) $ th lower order $\lambda
_{g}^{\left( p,q\right) }\left( f\right) $ $\left( 0<\lambda _{g}^{\left(
p,q\right) }\left( f\right) <\infty \right) $ where $p$\ and $q$ are any two
positive integers. Then the \emph{relative }$\left( p,q\right) $\emph{\ th
weak type} of meromorphic function $f$ with respect to the entire function $%
g $ is defined as :%
\begin{equation*}
\tau _{g}^{\left( p,q\right) }\left( f\right) =\underset{r\rightarrow \infty 
}{\lim \inf }\frac{\log ^{\left[ p-1\right] }T_{g}^{-1}T_{f}\left( r\right) 
}{\left( \log ^{\left[ q-1\right] }r\right) ^{\lambda _{g}^{\left(
p,q\right) }\left( f\right) }}~.
\end{equation*}%
Similarly one can define another growth indicator $\overline{\tau }%
_{g}^{\left( p,q\right) }\left( f\right) $ in the following way:%
\begin{equation*}
\overline{\tau }_{g}^{\left( p,q\right) }\left( f\right) =\underset{%
r\rightarrow \infty }{\lim \sup }\frac{\log ^{\left[ p-1\right]
}T_{g}^{-1}T_{f}\left( r\right) }{\left( \log ^{\left[ q-1\right] }r\right)
^{\lambda _{g}^{\left( p,q\right) }\left( f\right) }}~.
\end{equation*}
\end{definition}

\qquad In this paper we wish to prove some results related to the growth
rates of entire and meromorphic functions on the basis of relative $\left(
p,q\right) $ th order and relative $\left( p,q\right) $ th type of a
meromorphic function with respect to an entire function for any two positive
integers $p$ and $q$. We use the standard notations and definitions of the
theory of entire and meromorphic functions which are available in \cite{4}
and \cite{12}$.$

\section{\textbf{Main Results}}

\qquad In this section we present some results which will be needed in the
sequel.

\begin{theorem}
\label{t1} Let $f$, $g$ be any two meromorphic functions and $h,$ $k$ be any
two entire functions such that $0<\lambda _{h}^{\left( m,q\right) }\left(
f\right) \leq \rho _{h}^{\left( m,q\right) }\left( f\right) <\infty $ and $%
0<\lambda _{k}^{\left( n,q\right) }\left( g\right) \leq \rho _{k}^{\left(
n,q\right) }\left( g\right) <\infty $ where $m,n$ and $p$ are any three
positive integers. Then%
\begin{multline*}
\frac{\lambda _{h}^{\left( m,q\right) }\left( f\right) }{\rho _{k}^{\left(
n,q\right) }\left( g\right) }\leq \underset{r\rightarrow +\infty }{%
\underline{\lim }}\frac{\log ^{\left[ m\right] }T_{h}^{-1}T_{f}\left(
r\right) }{\log ^{\left[ n\right] }T_{k}^{-1}T_{g}\left( r\right) }\leq 
\frac{\lambda _{h}^{\left( m,q\right) }\left( f\right) }{\lambda
_{k}^{\left( n,q\right) }\left( g\right) } \\
\leq \underset{r\rightarrow +\infty }{\overline{\lim }}\frac{\log ^{\left[ m%
\right] }T_{h}^{-1}T_{f}\left( r\right) }{\log ^{\left[ n\right]
}T_{k}^{-1}T_{g}\left( r\right) }\leq \frac{\rho _{h}^{\left( m,q\right)
}\left( f\right) }{\lambda _{k}^{\left( n,q\right) }\left( g\right) }~.
\end{multline*}
\end{theorem}

\begin{proof}
From the definitions of $\lambda _{h}^{\left( m,q\right) }\left( f\right) $
and $\rho _{k}^{\left( n,q\right) }\left( g\right) ,$ we get for arbitrary
positive $\varepsilon $ and for all sufficiently large values of $r$ that%
\begin{equation}
\log ^{\left[ m\right] }T_{h}^{-1}T_{f}\left( r\right) \geqslant \left(
\lambda _{h}^{\left( m,q\right) }\left( f\right) -\varepsilon \right) \log ^{%
\left[ q\right] }r  \label{5.11}
\end{equation}%
and%
\begin{equation}
\log ^{\left[ n\right] }T_{k}^{-1}T_{g}\left( r\right) \leq \left( \rho
_{k}^{\left( n,q\right) }\left( g\right) +\varepsilon \right) \log ^{\left[ q%
\right] }r~.  \label{5.12}
\end{equation}

\qquad Now from $\left( \ref{5.11}\right) $ and $\left( \ref{5.12}\right) ,$
it follows for all sufficiently large values of $r$ that%
\begin{equation*}
\frac{\log ^{\left[ m\right] }T_{h}^{-1}T_{f}\left( r\right) }{\log ^{\left[
n\right] }T_{k}^{-1}T_{g}\left( r\right) }\geqslant \frac{\left( \lambda
_{h}^{\left( m,q\right) }\left( f\right) -\varepsilon \right) \log ^{\left[ q%
\right] }r}{\left( \rho _{k}^{\left( n,q\right) }\left( g\right)
+\varepsilon \right) \log ^{\left[ q\right] }r}~.
\end{equation*}

\qquad As $\varepsilon \left( >0\right) $ is arbitrary, we obtain that%
\begin{equation}
\underset{r\rightarrow +\infty }{\underline{\lim }}\frac{\log ^{\left[ m%
\right] }T_{h}^{-1}T_{f}\left( r\right) }{\log ^{\left[ n\right]
}T_{k}^{-1}T_{g}\left( r\right) }\geqslant \frac{\lambda _{h}^{\left(
m,q\right) }\left( f\right) }{\rho _{k}^{\left( n,q\right) }\left( g\right) }%
~.  \label{5.13}
\end{equation}

\qquad Again for a sequence of values of $r$ tending to infinity, we get that%
\begin{equation}
\log ^{\left[ m\right] }T_{h}^{-1}T_{f}\left( r\right) \leq \left( \lambda
_{h}^{\left( m,q\right) }\left( f\right) +\varepsilon \right) \log ^{\left[ q%
\right] }r  \label{5.14}
\end{equation}%
and for all sufficiently large values of $r$,%
\begin{equation}
\log ^{\left[ n\right] }T_{k}^{-1}T_{g}\left( r\right) \geqslant \left(
\lambda _{k}^{\left( n,q\right) }\left( g\right) -\varepsilon \right) \log ^{%
\left[ q\right] }r~.  \label{5.15}
\end{equation}

\qquad Combining $\left( \ref{5.14}\right) $ and $\left( \ref{5.15}\right) ,$
we obtain for a sequence of values of $r$ tending to infinity that%
\begin{equation*}
\frac{\log ^{\left[ m\right] }T_{h}^{-1}T_{f}\left( r\right) }{\log ^{\left[
n\right] }T_{k}^{-1}T_{g}\left( r\right) }\leq \frac{\left( \lambda
_{h}^{\left( m,q\right) }\left( f\right) +\varepsilon \right) \log ^{\left[ q%
\right] }r}{\left( \lambda _{k}^{\left( n,q\right) }\left( g\right)
-\varepsilon \right) \log ^{\left[ q\right] }r}~.
\end{equation*}

\qquad Since $\varepsilon \left( >0\right) $ is arbitrary, it follows that%
\begin{equation}
\underset{r\rightarrow +\infty }{\underline{\lim }}\frac{\log ^{\left[ m%
\right] }T_{h}^{-1}T_{f}\left( r\right) }{\log ^{\left[ n\right]
}T_{k}^{-1}T_{g}\left( r\right) }\leq \frac{\lambda _{h}^{\left( m,q\right)
}\left( f\right) }{\lambda _{k}^{\left( n,q\right) }\left( g\right) }~.
\label{5.16}
\end{equation}

\qquad Also for a sequence of values of $r$ tending to infinity, we get that%
\begin{equation}
\log ^{\left[ n\right] }T_{k}^{-1}T_{g}\left( r\right) \leq \left( \lambda
_{k}^{\left( n,q\right) }\left( g\right) +\varepsilon \right) \log ^{\left[ q%
\right] }r~.  \label{5.17}
\end{equation}

\qquad Now from $\left( \ref{5.11}\right) $ and $\left( \ref{5.17}\right) ,$
we obtain for a sequence of values of $r$ tending to infinity that%
\begin{equation*}
\frac{\log ^{\left[ m\right] }T_{h}^{-1}T_{f}\left( r\right) }{\log ^{\left[
n\right] }T_{k}^{-1}T_{g}\left( r\right) }\geq \frac{\left( \lambda
_{h}^{\left( m,q\right) }\left( f\right) -\varepsilon \right) \log ^{\left[ q%
\right] }r}{\left( \lambda _{k}^{\left( n,q\right) }\left( g\right)
+\varepsilon \right) \log ^{\left[ q\right] }r}~.
\end{equation*}

\qquad As $\varepsilon \left( >0\right) $ is arbitrary, we get from above
that%
\begin{equation}
\underset{r\rightarrow +\infty }{\overline{\lim }}\frac{\log ^{\left[ m%
\right] }T_{h}^{-1}T_{f}\left( r\right) }{\log ^{\left[ n\right]
}T_{k}^{-1}T_{g}\left( r\right) }\geq \frac{\lambda _{h}^{\left( m,q\right)
}\left( f\right) }{\lambda _{k}^{\left( n,q\right) }\left( g\right) }~.
\label{5.18}
\end{equation}

\qquad Also for all sufficiently large values of $r$,%
\begin{equation}
\log ^{\left[ m\right] }T_{h}^{-1}T_{f}\left( r\right) \leq \left( \rho
_{h}^{\left( m,q\right) }\left( f\right) +\varepsilon \right) \log ^{\left[ q%
\right] }r~.  \label{5.19}
\end{equation}

\qquad So it follows from $\left( \ref{5.15}\right) $ and $\left( \ref{5.19}%
\right) ,$ for all sufficiently large values of $r$ that%
\begin{equation*}
\frac{\log ^{\left[ m\right] }T_{h}^{-1}T_{f}\left( r\right) }{\log ^{\left[
n\right] }T_{k}^{-1}T_{g}\left( r\right) }\leq \frac{\left( \rho
_{h}^{\left( m,q\right) }\left( f\right) +\varepsilon \right) \log ^{\left[ q%
\right] }r}{\left( \lambda _{k}^{\left( n,q\right) }\left( g\right)
-\varepsilon \right) \log ^{\left[ q\right] }r}~.
\end{equation*}

\qquad Since $\varepsilon \left( >0\right) $ is arbitrary, we obtain that%
\begin{equation}
\underset{r\rightarrow +\infty }{\overline{\lim }}\frac{\log ^{\left[ m%
\right] }T_{h}^{-1}T_{f}\left( r\right) }{\log ^{\left[ n\right]
}T_{k}^{-1}T_{g}\left( r\right) }\leq \frac{\rho _{h}^{\left( m,q\right)
}\left( f\right) }{\lambda _{k}^{\left( n,q\right) }\left( g\right) }~.
\label{5.20}
\end{equation}%
Thus the theorem follows from $\left( \ref{5.13}\right) ,\left( \ref{5.16}%
\right) ,\left( \ref{5.18}\right) $ and $\left( \ref{5.20}\right) .$
\end{proof}

\begin{theorem}
\label{t3} Let $f$, $g$ be any two meromorphic functions and $h,$ $k$ be any
two entire functions such that $0<\rho _{h}^{\left( m,q\right) }\left(
f\right) <\infty $ and $0<\rho _{k}^{\left( n,q\right) }\left( g\right)
<\infty $ where $m,n$ and $p$ are any three positive integers. Then%
\begin{equation*}
\underset{r\rightarrow +\infty }{\underline{\lim }}\frac{\log ^{\left[ m%
\right] }T_{h}^{-1}T_{f}\left( r\right) }{\log ^{\left[ n\right]
}T_{k}^{-1}T_{g}\left( r\right) }\leq \frac{\rho _{h}^{\left( m,q\right)
}\left( f\right) }{\rho _{k}^{\left( n,q\right) }\left( g\right) }\leq 
\underset{r\rightarrow +\infty }{\overline{\lim }}\frac{\log ^{\left[ m%
\right] }T_{h}^{-1}T_{f}\left( r\right) }{\log ^{\left[ n\right]
}T_{k}^{-1}T_{g}\left( r\right) }~.
\end{equation*}
\end{theorem}

\begin{proof}
From the definition of $\rho _{k}^{\left( n,q\right) }\left( g\right) ,$ we
get for a sequence of values of $r$ tending to infinity that%
\begin{equation}
\log ^{\left[ n\right] }T_{k}^{-1}T_{g}\left( r\right) \geqslant \left( \rho
_{k}^{\left( n,q\right) }\left( g\right) -\varepsilon \right) \log ^{\left[ q%
\right] }r~.  \label{5.21}
\end{equation}

\qquad Now from $\left( \ref{5.19}\right) $ and $\left( \ref{5.21}\right) ,$
we get for a sequence of values of $r$ tending to infinity that%
\begin{equation*}
\frac{\log ^{\left[ m\right] }T_{h}^{-1}T_{f}\left( r\right) }{\log ^{\left[
n\right] }T_{k}^{-1}T_{g}\left( r\right) }\leq \frac{\left( \rho
_{h}^{\left( m,q\right) }\left( f\right) +\varepsilon \right) \log ^{\left[ q%
\right] }r}{\left( \rho _{k}^{\left( n,q\right) }\left( g\right)
-\varepsilon \right) \log ^{\left[ q\right] }r}~.
\end{equation*}

\qquad As $\varepsilon \left( >0\right) $ is arbitrary, we obtain that%
\begin{equation}
\underset{r\rightarrow +\infty }{\underline{\lim }}\frac{\log ^{\left[ m%
\right] }T_{h}^{-1}T_{f}\left( r\right) }{\log ^{\left[ n\right]
}T_{k}^{-1}T_{g}\left( r\right) }\leq \frac{\rho _{h}^{\left( m,q\right)
}\left( f\right) }{\rho _{k}^{\left( n,q\right) }\left( g\right) }~.
\label{5.22}
\end{equation}

\qquad Also for a sequence of values of $r$ tending to infinity, we have%
\begin{equation}
\log ^{\left[ m\right] }T_{h}^{-1}T_{f}\left( r\right) \geqslant \left( \rho
_{h}^{\left( m,q\right) }\left( f\right) -\varepsilon \right) \log ^{\left[ q%
\right] }r~.  \label{5.23}
\end{equation}

\qquad So combining $\left( \ref{5.12}\right) $ and $\left( \ref{5.23}%
\right) ,$ we get for a sequence of values of $r$ tending to infinity that%
\begin{equation*}
\frac{\log ^{\left[ m\right] }T_{h}^{-1}T_{f}\left( r\right) }{\log ^{\left[
n\right] }T_{k}^{-1}T_{g}\left( r\right) }\geqslant \frac{\left( \rho
_{h}^{\left( m,q\right) }\left( f\right) -\varepsilon \right) \log ^{\left[ q%
\right] }r}{\left( \rho _{k}^{\left( n,q\right) }\left( g\right)
+\varepsilon \right) \log ^{\left[ q\right] }r}~.
\end{equation*}

\qquad Since $\varepsilon \left( >0\right) $ is arbitrary, it follows that%
\begin{equation}
\underset{r\rightarrow +\infty }{\overline{\lim }}\frac{\log ^{\left[ m%
\right] }T_{h}^{-1}T_{f}\left( r\right) }{\log ^{\left[ n\right]
}T_{k}^{-1}T_{g}\left( r\right) }\geqslant \frac{\rho _{h}^{\left(
m,q\right) }\left( f\right) }{\rho _{k}^{\left( n,q\right) }\left( g\right) }%
~.  \label{5.24}
\end{equation}

\qquad Thus the theorem follows from $\left( \ref{5.22}\right) $ and $\left( %
\ref{5.24}\right) .$
\end{proof}

\qquad The following theorem is a natural consequence of Theorem \ref{t1}
and Theorem \ref{t3}.

\begin{theorem}
\label{t5} Let $f$, $g$ be any two meromorphic functions and $h,$ $k$ be any
two entire functions such that $0<\lambda _{h}^{\left( m,q\right) }\left(
f\right) \leq \rho _{h}^{\left( m,q\right) }\left( f\right) <\infty $ and $%
0<\lambda _{k}^{\left( n,q\right) }\left( g\right) \leq \rho _{k}^{\left(
n,q\right) }\left( g\right) <\infty $ where $m,n$ and $p$ are any three
positive integers. Then%
\begin{multline*}
\underset{r\rightarrow +\infty }{\underline{\lim }}\frac{\log ^{\left[ m%
\right] }T_{h}^{-1}T_{f}\left( r\right) }{\log ^{\left[ n\right]
}T_{k}^{-1}T_{g}\left( r\right) }\leq \min \left\{ \frac{\lambda
_{h}^{\left( m,q\right) }\left( f\right) }{\lambda _{k}^{\left( n,q\right)
}\left( g\right) },\frac{\rho _{h}^{\left( m,q\right) }\left( f\right) }{%
\rho _{k}^{\left( n,q\right) }\left( g\right) }\right\} \\
\leq \max \left\{ \frac{\lambda _{h}^{\left( m,q\right) }\left( f\right) }{%
\lambda _{k}^{\left( n,q\right) }\left( g\right) },\frac{\rho _{h}^{\left(
m,q\right) }\left( f\right) }{\rho _{k}^{\left( n,q\right) }\left( g\right) }%
\right\} \leq \underset{r\rightarrow +\infty }{\overline{\lim }}\frac{\log ^{%
\left[ m\right] }T_{h}^{-1}T_{f}\left( r\right) }{\log ^{\left[ n\right]
}T_{k}^{-1}T_{g}\left( r\right) }~.
\end{multline*}
\end{theorem}

\qquad The proof is omitted.

\begin{theorem}
\label{t1x} Let $f$, $g$ be any two meromorphic functions and $h,$ $k$ be
any two entire functions such that $0<$ $\overline{\sigma }_{h}^{\left(
m,q\right) }\left( f\right) $ $\leq $ $\sigma _{h}^{\left( m,q\right)
}\left( f\right) $ $<\infty $, $0<$ $\overline{\sigma }_{k}^{\left(
n,q\right) }\left( g\right) $ $\leq $ $\sigma _{k}^{\left( n,q\right)
}\left( g\right) $ $<\infty $ and $\rho _{h}^{\left( m,q\right) }\left(
f\right) $ $=$ $\rho _{k}^{\left( n,q\right) }\left( g\right) $ where $m,n$
and $p$ are any three positive integers. Then%
\begin{multline*}
\frac{\overline{\sigma }_{h}^{\left( m,q\right) }\left( f\right) }{\sigma
_{k}^{\left( n,q\right) }\left( g\right) }\leq \underset{r\rightarrow
+\infty }{\underline{\lim }}\frac{\log ^{\left[ m-1\right]
}T_{h}^{-1}T_{f}\left( r\right) }{\log ^{\left[ n-1\right]
}T_{k}^{-1}T_{g}\left( r\right) }\leq \frac{\overline{\sigma }_{h}^{\left(
m,q\right) }\left( f\right) }{\overline{\sigma }_{k}^{\left( n,q\right)
}\left( g\right) } \\
\leq \underset{r\rightarrow +\infty }{\overline{\lim }}\frac{\log ^{\left[
m-1\right] }T_{h}^{-1}T_{f}\left( r\right) }{\log ^{\left[ n-1\right]
}T_{k}^{-1}T_{g}\left( r\right) }\leq \frac{\sigma _{h}^{\left( m,q\right)
}\left( f\right) }{\overline{\sigma }_{k}^{\left( n,q\right) }\left(
g\right) }~.
\end{multline*}
\end{theorem}

\begin{proof}
From the definition of $\overline{\sigma }_{h}^{\left( m,q\right) }\left(
f\right) $ and $\sigma _{k}^{\left( n,q\right) }\left( g\right) $, we have
for arbitrary positive $\varepsilon $ and for all sufficiently large values
of $r$ that%
\begin{equation}
\log ^{\left[ m-1\right] }T_{h}^{-1}T_{f}\left( r\right) \geq \left( 
\overline{\sigma }_{h}^{\left( m,q\right) }\left( f\right) -\varepsilon
\right) \left[ \log ^{\left[ q-1\right] }r\right] ^{\rho _{h}^{\left(
m,q\right) }\left( f\right) },  \label{55.11}
\end{equation}%
and%
\begin{equation}
\log ^{\left[ n-1\right] }T_{k}^{-1}T_{g}\left( r\right) \leq \left( \sigma
_{k}^{\left( n,q\right) }\left( g\right) +\varepsilon \right) \left[ \log ^{%
\left[ q-1\right] }r\right] ^{\rho _{k}^{\left( n,q\right) }\left( g\right)
}~\text{.}  \label{55.12}
\end{equation}%
Now from $\left( \ref{55.11}\right) $, $\left( \ref{55.12}\right) $ and the
condition $\rho _{h}^{\left( m,q\right) }\left( f\right) $ $=$ $\rho
_{k}^{\left( n,q\right) }\left( g\right) ,$ it follows that for all
sufficiently large values of $r$ that%
\begin{equation*}
\frac{\log ^{\left[ m-1\right] }T_{h}^{-1}T_{f}\left( r\right) }{\log ^{%
\left[ n-1\right] }T_{k}^{-1}T_{g}\left( r\right) }\geqslant \frac{\overline{%
\sigma }_{h}^{\left( m,q\right) }\left( f\right) -\varepsilon }{\sigma
_{k}^{\left( n,q\right) }\left( g\right) +\varepsilon }.
\end{equation*}%
As $\varepsilon \left( >0\right) $ is arbitrary , we obtain from above that%
\begin{equation}
\underset{r\rightarrow +\infty }{\underline{\lim }}\frac{\log ^{\left[ m-1%
\right] }T_{h}^{-1}T_{f}\left( r\right) }{\log ^{\left[ n-1\right]
}T_{k}^{-1}T_{g}\left( r\right) }\geqslant \frac{\overline{\sigma }%
_{h}^{\left( m,q\right) }\left( f\right) }{\sigma _{k}^{\left( n,q\right)
}\left( g\right) }~.  \label{55.13}
\end{equation}%
Again for a sequence of values of $r$ tending to infinity, we get that%
\begin{equation}
\log ^{\left[ m-1\right] }T_{h}^{-1}T_{f}\left( r\right) \leq \left( 
\overline{\sigma }_{h}^{\left( m,q\right) }\left( f\right) +\varepsilon
\right) \left[ \log ^{\left[ q-1\right] }r\right] ^{\rho _{h}^{\left(
m,q\right) }\left( f\right) }  \label{55.14}
\end{equation}%
and for all sufficiently large values of $r$,%
\begin{equation}
\log ^{\left[ n-1\right] }T_{k}^{-1}T_{g}\left( r\right) \geq \left( 
\overline{\sigma }_{k}^{\left( n,q\right) }\left( g\right) -\varepsilon
\right) \left[ \log ^{\left[ q-1\right] }r\right] ^{\rho _{k}^{\left(
n,q\right) }\left( g\right) }~\text{.}  \label{55.15}
\end{equation}%
Combining $\left( \ref{55.14}\right) $ and $\left( \ref{55.15}\right) $ and
the condition $\rho _{h}^{\left( m,q\right) }\left( f\right) $ $=$ $\rho
_{k}^{\left( n,q\right) }\left( g\right) ,$ we get for a sequence of values
of $r$ tending to infinity that%
\begin{equation*}
\frac{\log ^{\left[ m-1\right] }T_{h}^{-1}T_{f}\left( r\right) }{\log ^{%
\left[ n-1\right] }T_{k}^{-1}T_{g}\left( r\right) }\leq \frac{\overline{%
\sigma }_{h}^{\left( m,q\right) }\left( f\right) +\varepsilon }{\overline{%
\sigma }_{k}^{\left( n,q\right) }\left( g\right) -\varepsilon }~.
\end{equation*}%
Since $\varepsilon \left( >0\right) $ is arbitrary, it follows from above
that%
\begin{equation}
\underset{r\rightarrow +\infty }{\underline{\lim }}\frac{\log ^{\left[ m-1%
\right] }T_{h}^{-1}T_{f}\left( r\right) }{\log ^{\left[ n-1\right]
}T_{k}^{-1}T_{g}\left( r\right) }\leq \frac{\overline{\sigma }_{h}^{\left(
m,q\right) }\left( f\right) }{\overline{\sigma }_{k}^{\left( n,q\right)
}\left( g\right) }~.  \label{55.16}
\end{equation}%
Also for a sequence of values of $r$ tending to infinity\ it follows that%
\begin{equation}
\log ^{\left[ n-1\right] }T_{k}^{-1}T_{g}\left( r\right) \leq \left( 
\overline{\sigma }_{k}^{\left( n,q\right) }\left( g\right) +\varepsilon
\right) \left[ \log ^{\left[ q-1\right] }r\right] ^{\rho _{k}^{\left(
n,q\right) }\left( g\right) }~.  \label{55.17}
\end{equation}%
Now from $\left( \ref{55.11}\right) $, $\left( \ref{55.17}\right) $ and the
condition $\rho _{h}^{\left( m,q\right) }\left( f\right) $ $=$ $\rho
_{k}^{\left( n,q\right) }\left( g\right) ,$ we obtain for a sequence of
values of $r$ tending to infinity that%
\begin{equation*}
\frac{\log ^{\left[ m-1\right] }T_{h}^{-1}T_{f}\left( r\right) }{\log ^{%
\left[ n-1\right] }T_{k}^{-1}T_{g}\left( r\right) }\geq \frac{\overline{%
\sigma }_{h}^{\left( m,q\right) }\left( f\right) -\varepsilon }{\overline{%
\sigma }_{k}^{\left( n,q\right) }\left( g\right) +\varepsilon }~\text{.}
\end{equation*}%
As $\varepsilon \left( >0\right) $ is arbitrary, we get from above that%
\begin{equation}
\underset{r\rightarrow +\infty }{\overline{\lim }}\frac{\log ^{\left[ m-1%
\right] }T_{h}^{-1}T_{f}\left( r\right) }{\log ^{\left[ n-1\right]
}T_{k}^{-1}T_{g}\left( r\right) }\geq \frac{\overline{\sigma }_{h}^{\left(
m,q\right) }\left( f\right) }{\overline{\sigma }_{k}^{\left( n,q\right)
}\left( g\right) }~.  \label{55.18}
\end{equation}%
Also for all sufficiently large values of $r$, we get that%
\begin{equation}
\log ^{\left[ m-1\right] }T_{h}^{-1}T_{f}\left( r\right) \leq \left( \sigma
_{h}^{\left( m,q\right) }\left( f\right) +\varepsilon \right) \left[ \log ^{%
\left[ q-1\right] }r\right] ^{\rho _{h}^{\left( m,q\right) }\left( f\right)
}~.  \label{55.19}
\end{equation}%
In view of the condition $\rho _{h}^{\left( m,q\right) }\left( f\right) $ $=$
$\rho _{k}^{\left( n,q\right) }\left( g\right) ,$ it follows from $\left( %
\ref{55.15}\right) $ and $\left( \ref{55.19}\right) $ for all sufficiently
large values of $r$ that%
\begin{equation*}
\frac{\log ^{\left[ m-1\right] }T_{h}^{-1}T_{f}\left( r\right) }{\log ^{%
\left[ n-1\right] }T_{k}^{-1}T_{g}\left( r\right) }\leq \frac{\sigma
_{h}^{\left( m,q\right) }\left( f\right) +\varepsilon }{\overline{\sigma }%
_{k}^{\left( n,q\right) }\left( g\right) -\varepsilon }~.
\end{equation*}%
Since $\varepsilon \left( >0\right) $ is arbitrary, we obtain that%
\begin{equation}
\underset{r\rightarrow +\infty }{\overline{\lim }}\frac{\log ^{\left[ m-1%
\right] }T_{h}^{-1}T_{f}\left( r\right) }{\log ^{\left[ n-1\right]
}T_{k}^{-1}T_{g}\left( r\right) }\leq \frac{\sigma _{h}^{\left( m,q\right)
}\left( f\right) }{\overline{\sigma }_{k}^{\left( n,q\right) }\left(
g\right) }~.  \label{55.20}
\end{equation}%
Thus the theorem follows from $\left( \ref{55.13}\right) ,$ $\left( \ref%
{55.16}\right) ,$ $\left( \ref{55.18}\right) $ and $\left( \ref{55.20}%
\right) .$
\end{proof}

\begin{theorem}
\label{t3x} Let $f$, $g$ be any two meromorphic functions and $h,$ $k$ be
any two entire functions such that $0<$ $\Delta _{h}^{\left( m,q\right)
}\left( f\right) $ $<\infty $, $0<$ $\Delta _{k}^{\left( n,q\right) }\left(
g\right) $ $<\infty $ and $\rho _{h}^{\left( m,q\right) }\left( f\right) $ $%
= $ $\rho _{k}^{\left( n,q\right) }\left( g\right) $ where $m,n$ and $p$ are
any three positive integers. Then%
\begin{equation*}
\underset{r\rightarrow +\infty }{\underline{\lim }}\frac{\log ^{\left[ m-1%
\right] }T_{h}^{-1}T_{f}\left( r\right) }{\log ^{\left[ n-1\right]
}T_{k}^{-1}T_{g}\left( r\right) }\leq \frac{\Delta _{h}^{\left( m,q\right)
}\left( f\right) }{\Delta _{k}^{\left( n,q\right) }\left( g\right) }\leq 
\underset{r\rightarrow +\infty }{\overline{\lim }}\frac{\log ^{\left[ m-1%
\right] }T_{h}^{-1}T_{f}\left( r\right) }{\log ^{\left[ n-1\right]
}T_{k}^{-1}T_{g}\left( r\right) }~.
\end{equation*}
\end{theorem}

\begin{proof}
From the definition of $\Delta _{k}^{\left( n,q\right) }\left( g\right) ,$
we get for a sequence of values of $r$ tending to infinity that%
\begin{equation}
\log ^{\left[ n-1\right] }T_{k}^{-1}T_{g}\left( r\right) \geq \left( \Delta
_{k}^{\left( n,q\right) }\left( g\right) -\varepsilon \right) \left[ \log ^{%
\left[ q-1\right] }r\right] ^{\rho _{k}^{\left( n,q\right) }\left( g\right)
}~.  \label{55.21}
\end{equation}%
Now from $\left( \ref{55.19}\right) $, $\left( \ref{55.21}\right) $ and the
condition $\rho _{h}^{\left( m,q\right) }\left( f\right) $ $=$ $\rho
_{k}^{\left( n,q\right) }\left( g\right) ,$ it follows for a sequence of
values of $r$ tending to infinity that%
\begin{equation*}
\frac{\log ^{\left[ m-1\right] }T_{h}^{-1}T_{f}\left( r\right) }{\log ^{%
\left[ n-1\right] }T_{k}^{-1}T_{g}\left( r\right) }\leq \frac{\Delta
_{h}^{\left( m,q\right) }\left( f\right) +\varepsilon }{\Delta _{k}^{\left(
n,q\right) }\left( g\right) -\varepsilon }~.
\end{equation*}%
As $\varepsilon \left( >0\right) $ is arbitrary, we obtain that%
\begin{equation}
\underset{r\rightarrow +\infty }{\underline{\lim }}\frac{\log ^{\left[ m-1%
\right] }T_{h}^{-1}T_{f}\left( r\right) }{\log ^{\left[ n-1\right]
}T_{k}^{-1}T_{g}\left( r\right) }\leq \frac{\Delta _{h}^{\left( m,q\right)
}\left( f\right) }{\Delta _{k}^{\left( n,q\right) }\left( g\right) }~.
\label{55.22}
\end{equation}%
Again for a sequence of values of $r$ tending to infinity that%
\begin{equation}
\log ^{\left[ m-1\right] }T_{h}^{-1}T_{f}\left( r\right) \geqslant \left(
\Delta _{h}^{\left( m,q\right) }\left( f\right) -\varepsilon \right) \left[
\log ^{\left[ q-1\right] }r\right] ^{\rho _{h}^{\left( m,q\right) }\left(
f\right) }~.  \label{55.23}
\end{equation}%
So combining $\left( \ref{55.12}\right) $ and $\left( \ref{55.23}\right) $
and in view of the condition $\rho _{h}^{\left[ m\right] }\left( f\right) $ $%
=$ $\rho _{k}^{\left( n,q\right) }\left( g\right) ,$ we get for a sequence
of values of $r$ tending to infinity that%
\begin{equation*}
\frac{\log ^{\left[ m-1\right] }T_{h}^{-1}T_{f}\left( r\right) }{\log ^{%
\left[ n-1\right] }T_{k}^{-1}T_{g}\left( r\right) }\geqslant \frac{\Delta
_{h}^{\left( m,q\right) }\left( f\right) -\varepsilon }{\Delta _{k}^{\left(
n,q\right) }\left( g\right) +\varepsilon }~.
\end{equation*}%
Since $\varepsilon \left( >0\right) $ is arbitrary, it follows that%
\begin{equation}
\underset{r\rightarrow +\infty }{\overline{\lim }}\frac{\log ^{\left[ m-1%
\right] }T_{h}^{-1}T_{f}\left( r\right) }{\log ^{\left[ n-1\right]
}T_{k}^{-1}T_{g}\left( r\right) }\geqslant \frac{\Delta _{h}^{\left(
m,q\right) }\left( f\right) }{\Delta _{k}^{\left( n,q\right) }\left(
g\right) }~.  \label{55.24}
\end{equation}%
Thus the theorem follows from $\left( \ref{55.22}\right) $ and $\left( \ref%
{55.24}\right) .$
\end{proof}

\qquad The following theorem is a natural consequence of Theorem \ref{t1x}
and Theorem \ref{t3x}:

\begin{theorem}
\label{t5x} Let $f$, $g$ be any two meromorphic functions and $h,$ $k$ be
any two entire functions such that $0<$ $\overline{\sigma }_{h}^{\left(
m,q\right) }\left( f\right) $ $\leq $ $\sigma _{h}^{\left( m,q\right)
}\left( f\right) $ $<\infty $, $0<$ $\overline{\sigma }_{k}^{\left(
n,q\right) }\left( g\right) $ $\leq $ $\sigma _{k}^{\left( n,q\right)
}\left( g\right) $ $<\infty $ and $\rho _{h}^{\left( m,q\right) }\left(
f\right) $ $=$ $\rho _{k}^{\left( n,q\right) }\left( g\right) $ where $m,n$
and $p$ are any three positive integers. Then%
\begin{multline*}
\underset{r\rightarrow +\infty }{\underline{\lim }}\frac{\log ^{\left[ m-1%
\right] }T_{h}^{-1}T_{f}\left( r\right) }{\log ^{\left[ n-1\right]
}T_{k}^{-1}T_{g}\left( r\right) }\leq \min \left\{ \frac{\overline{\sigma }%
_{h}^{\left( m,q\right) }\left( f\right) }{\overline{\sigma }_{k}^{\left(
n,q\right) }\left( g\right) },\frac{\sigma _{h}^{\left( m,q\right) }\left(
f\right) }{\sigma _{k}^{\left( n,q\right) }\left( g\right) }\right\} \\
\leq \max \left\{ \frac{\overline{\sigma }_{h}^{\left( m,q\right) }\left(
f\right) }{\overline{\sigma }_{k}^{\left( n,q\right) }\left( g\right) },%
\frac{\sigma _{h}^{\left( m,q\right) }\left( f\right) }{\sigma _{k}^{\left(
n,q\right) }\left( g\right) }\right\} \leq \underset{r\rightarrow +\infty }{%
\overline{\lim }}\frac{\log ^{\left[ m-1\right] }T_{h}^{-1}T_{f}\left(
r\right) }{\log ^{\left[ n-1\right] }T_{k}^{-1}T_{g}\left( r\right) }~.
\end{multline*}
\end{theorem}

\qquad Now in the line of Theorem \ref{t1x}, Theorem \ref{t3x} and Theorem %
\ref{t5x} respectively, one can easily prove the following six theorems
using the notion of $(p,q)$-th relative weak type and therefore their proofs
are omitted.

\begin{theorem}
\label{t7x} Let $f$, $g$ be any two meromorphic functions and $h,$ $k$ be
any two entire functions such that $0<$ $\tau _{h}^{\left( m,q\right)
}\left( f\right) $ $\leq $ $\overline{\tau }_{h}^{\left( m,q\right) }\left(
f\right) $ $<\infty $, $0<$ $\tau _{k}^{\left( n,q\right) }\left( g\right) $ 
$\leq $ $\overline{\tau }_{k}^{\left( n,q\right) }\left( g\right) $ $<\infty 
$ and $\lambda _{h}^{\left( m,q\right) }\left( f\right) $ $=$ $\lambda
_{k}^{\left( n,q\right) }\left( g\right) $ where $m,n$ and $p$ are any three
positive integers. Then%
\begin{multline*}
\frac{\tau _{h}^{\left( m,q\right) }\left( f\right) }{\overline{\tau }%
_{k}^{\left( n,q\right) }\left( g\right) }\leq \underset{r\rightarrow
+\infty }{\underline{\lim }}\frac{\log ^{\left[ m-1\right]
}T_{h}^{-1}T_{f}\left( r\right) }{\log ^{\left[ n-1\right]
}T_{k}^{-1}T_{g}\left( r\right) }\leq \frac{\tau _{h}^{\left( m,q\right)
}\left( f\right) }{\tau _{k}^{\left( n,q\right) }\left( g\right) } \\
\leq \underset{r\rightarrow +\infty }{\overline{\lim }}\frac{\log ^{\left[
m-1\right] }T_{h}^{-1}T_{f}\left( r\right) }{\log ^{\left[ n-1\right]
}T_{k}^{-1}T_{g}\left( r\right) }\leq \frac{\overline{\tau }_{h}^{\left(
m,q\right) }\left( f\right) }{\tau _{k}^{\left( n,q\right) }\left( g\right) }%
~.
\end{multline*}
\end{theorem}

\begin{theorem}
\label{t8} Let $f$, $g$ be any two meromorphic functions and $h,$ $k$ be any
two entire functions such that $0<$ $\overline{\tau }_{h}^{\left( m,q\right)
}\left( f\right) $ $<\infty $, $0<$ $\overline{\tau }_{k}^{\left( n,q\right)
}\left( g\right) $ $<\infty $ and $\lambda _{h}^{\left( m,q\right) }\left(
f\right) $ $=$ $\lambda _{k}^{\left( n,q\right) }\left( g\right) $ where $%
m,n $ and $p$ are any three positive integers. Then%
\begin{equation*}
\underset{r\rightarrow +\infty }{\underline{\lim }}\frac{\log ^{\left[ m-1%
\right] }T_{h}^{-1}T_{f}\left( r\right) }{\log ^{\left[ n-1\right]
}T_{k}^{-1}T_{g}\left( r\right) }\leq \frac{\overline{\tau }_{h}^{\left(
m,q\right) }\left( f\right) }{\overline{\tau }_{k}^{\left( n,q\right)
}\left( g\right) }\leq \underset{r\rightarrow +\infty }{\overline{\lim }}%
\frac{\log ^{\left[ m-1\right] }T_{h}^{-1}T_{f}\left( r\right) }{\log ^{%
\left[ n-1\right] }T_{k}^{-1}T_{g}\left( r\right) }~.
\end{equation*}
\end{theorem}

\begin{theorem}
\label{t9} Let $f$, $g$ be any two meromorphic functions and $h,$ $k$ be any
two entire functions such that $0<$ $\tau _{h}^{\left( m,q\right) }\left(
f\right) $ $\leq $ $\overline{\tau }_{h}^{\left( m,q\right) }\left( f\right) 
$ $<\infty $, $0<$ $\tau _{k}^{\left( n,q\right) }\left( g\right) $ $\leq $ $%
\overline{\tau }_{k}^{\left( n,q\right) }\left( g\right) $ $<\infty $ and $%
\lambda _{h}^{\left( m,q\right) }\left( f\right) $ $=$ $\lambda _{k}^{\left(
n,q\right) }\left( g\right) $ where $m,n$ and $p$ are any three positive
integers. Then%
\begin{multline*}
\underset{r\rightarrow +\infty }{\underline{\lim }}\frac{\log ^{\left[ m-1%
\right] }T_{h}^{-1}T_{f}\left( r\right) }{\log ^{\left[ n-1\right]
}T_{k}^{-1}T_{g}\left( r\right) }\leq \min \left\{ \frac{\tau _{h}^{\left(
m,q\right) }\left( f\right) }{\tau _{k}^{\left( n,q\right) }\left( g\right) }%
,\frac{\overline{\tau }_{h}^{\left( m,q\right) }\left( f\right) }{\overline{%
\tau }_{k}^{\left( n,q\right) }\left( g\right) }\right\} \\
\leq \max \left\{ \frac{\tau _{h}^{\left( m,q\right) }\left( f\right) }{\tau
_{k}^{\left( n,q\right) }\left( g\right) },\frac{\overline{\tau }%
_{h}^{\left( m,q\right) }\left( f\right) }{\overline{\tau }_{k}^{\left(
n,q\right) }\left( g\right) }\right\} \leq \underset{r\rightarrow +\infty }{%
\overline{\lim }}\frac{\log ^{\left[ m-1\right] }T_{h}^{-1}T_{f}\left(
r\right) }{\log ^{\left[ n-1\right] }T_{k}^{-1}T_{g}\left( r\right) }~.
\end{multline*}
\end{theorem}

\qquad We may now state the following theorems without their proofs based on 
$(p,q)$-th relative type and\ $(p,q)$-th\ relative weak type:

\begin{theorem}
\label{t13} Let $f$, $g$ be any two meromorphic functions and $h,$ $k$ be
any two entire functions such that $0<$ $\overline{\sigma }_{h}^{\left(
m,q\right) }\left( f\right) $ $\leq $ $\sigma _{h}^{\left( m,q\right)
}\left( f\right) $ $<\infty $, $0<$ $\tau _{k}^{\left( n,q\right) }\left(
g\right) $ $\leq $ $\overline{\tau }_{k}^{\left( n,q\right) }\left( g\right) 
$ $<\infty $ and $\rho _{h}^{\left( m,q\right) }\left( f\right) $ $=$ $%
\lambda _{k}^{\left( n,q\right) }\left( g\right) $ where $m,n$ and $p$ are
any three positive integers. Then%
\begin{multline*}
\frac{\overline{\sigma }_{h}^{\left( m,q\right) }\left( f\right) }{\overline{%
\tau }_{k}^{\left( n,q\right) }\left( g\right) }\leq \underset{r\rightarrow
+\infty }{\underline{\lim }}\frac{\log ^{\left[ m-1\right]
}T_{h}^{-1}T_{f}\left( r\right) }{\log ^{\left[ n-1\right]
}T_{k}^{-1}T_{g}\left( r\right) }\leq \frac{\overline{\sigma }_{h}^{\left(
m,q\right) }\left( f\right) }{\tau _{k}^{\left( n,q\right) }\left( g\right) }
\\
\leq \underset{r\rightarrow +\infty }{\overline{\lim }}\frac{\log ^{\left[
m-1\right] }T_{h}^{-1}T_{f}\left( r\right) }{\log ^{\left[ n-1\right]
}T_{k}^{-1}T_{g}\left( r\right) }\leq \frac{\sigma _{h}^{\left( m,q\right)
}\left( f\right) }{\tau _{k}^{\left( n,q\right) }\left( g\right) }~.
\end{multline*}
\end{theorem}

\begin{theorem}
\label{t14} Let $f$, $g$ be any two meromorphic functions and $h,$ $k$ be
any two entire functions such that $0<$ $\sigma _{h}^{\left( m,q\right)
}\left( f\right) $ $<\infty $, $0<$ $\overline{\tau }_{k}^{\left( n,q\right)
}\left( g\right) $ $<\infty $ and $\rho _{h}^{\left( m,q\right) }\left(
f\right) $ $=$ $\lambda _{k}^{\left( n,q\right) }\left( g\right) $ where $%
m,n $ and $p$ are any three positive integers. Then%
\begin{equation*}
\underset{r\rightarrow +\infty }{\underline{\lim }}\frac{\log ^{\left[ m-1%
\right] }T_{h}^{-1}T_{f}\left( r\right) }{\log ^{\left[ n-1\right]
}T_{k}^{-1}T_{g}\left( r\right) }\leq \frac{\sigma _{h}^{\left( m,q\right)
}\left( f\right) }{\overline{\tau }_{k}^{\left( n,q\right) }\left( g\right) }%
\leq \underset{r\rightarrow \infty }{\overline{\lim }}\frac{\log ^{\left[ m-1%
\right] }T_{h}^{-1}T_{f}\left( r\right) }{\log ^{\left[ n-1\right]
}T_{k}^{-1}T_{g}\left( r\right) }~.
\end{equation*}
\end{theorem}

\begin{theorem}
\label{t15} Let $f$, $g$ be any two meromorphic functions and $h,$ $k$ be
any two entire functions such that $0<$ $\overline{\sigma }_{h}^{\left(
m,q\right) }\left( f\right) $ $\leq $ $\sigma _{h}^{\left( m,q\right)
}\left( f\right) $ $<\infty $, $0<$ $\tau _{k}^{\left( n,q\right) }\left(
g\right) $ $\leq $ $\overline{\tau }_{k}^{\left( n,q\right) }\left( g\right) 
$ $<\infty $ and $\rho _{h}^{\left( m,q\right) }\left( f\right) $ $=$ $%
\lambda _{k}^{\left( n,q\right) }\left( g\right) $ where $m,n$ and $p$ are
any three positive integers. Then%
\begin{multline*}
\underset{r\rightarrow +\infty }{\underline{\lim }}\frac{\log ^{\left[ m-1%
\right] }T_{h}^{-1}T_{f}\left( r\right) }{\log ^{\left[ n-1\right]
}T_{k}^{-1}T_{g}\left( r\right) }\leq \min \left\{ \frac{\overline{\sigma }%
_{h}^{\left( m,q\right) }\left( f\right) }{\tau _{k}^{\left( n,q\right)
}\left( g\right) },\frac{\sigma _{h}^{\left( m,q\right) }\left( f\right) }{%
\overline{\tau }_{k}^{\left( n,q\right) }\left( g\right) }\right\} \\
\leq \max \left\{ \frac{\overline{\sigma }_{h}^{\left( m,q\right) }\left(
f\right) }{\tau _{k}^{\left( n,q\right) }\left( g\right) },\frac{\sigma
_{h}^{\left( m,q\right) }\left( f\right) }{\overline{\tau }_{k}^{\left(
n,q\right) }\left( g\right) }\right\} \leq \underset{r\rightarrow +\infty }{%
\overline{\lim }}\frac{\log ^{\left[ m-1\right] }T_{h}^{-1}T_{f}\left(
r\right) }{\log ^{\left[ n-1\right] }T_{k}^{-1}T_{g}\left( r\right) }~.
\end{multline*}
\end{theorem}

\begin{theorem}
\label{t16} Let $f$, $g$ be any two meromorphic functions and $h,$ $k$ be
any two entire functions such that $0<$ $\tau _{h}^{\left( m,q\right)
}\left( f\right) $ $\leq $ $\overline{\tau }_{h}^{\left( m,q\right) }\left(
f\right) $ $<\infty $, $0<$ $\overline{\sigma }_{k}^{\left( n,q\right)
}\left( g\right) $ $\leq $ $\sigma _{k}^{\left( n,q\right) }\left( g\right) $
$<\infty $ and $\lambda _{h}^{\left( m,q\right) }\left( f\right) $ $=$ $\rho
_{k}^{\left( n,q\right) }\left( g\right) $ where $m,n$ and $p$ are any three
positive integers. Then%
\begin{multline*}
\frac{\tau _{h}^{\left( m,q\right) }\left( f\right) }{\sigma _{k}^{\left(
n,q\right) }\left( g\right) }\leq \underset{r\rightarrow +\infty }{%
\underline{\lim }}\frac{\log ^{\left[ m-1\right] }T_{h}^{-1}T_{f}\left(
r\right) }{\log ^{\left[ n-1\right] }T_{k}^{-1}T_{g}\left( r\right) }\leq 
\frac{\tau _{h}^{\left( m,q\right) }\left( f\right) }{\overline{\sigma }%
_{k}^{\left( n,q\right) }\left( g\right) } \\
\leq \underset{r\rightarrow +\infty }{\overline{\lim }}\frac{\log ^{\left[
m-1\right] }T_{h}^{-1}T_{f}\left( r\right) }{\log ^{\left[ n-1\right]
}T_{k}^{-1}T_{g}\left( r\right) }\leq \frac{\overline{\tau }_{h}^{\left(
m,q\right) }\left( f\right) }{\overline{\sigma }_{k}^{\left( n,q\right)
}\left( g\right) }~.
\end{multline*}
\end{theorem}

\begin{theorem}
\label{t17} Let $f$, $g$ be any two meromorphic functions and $h,$ $k$ be
any two entire functions such that $0<$ $\overline{\tau }_{h}^{\left(
m,q\right) }\left( f\right) $ $<\infty $, $0<$ $\sigma _{k}^{\left(
n,q\right) }\left( g\right) $ $<\infty $ and $\lambda _{h}^{\left(
m,q\right) }\left( f\right) $ $=$ $\rho _{k}^{\left( n,q\right) }\left(
g\right) $ where $m,n$ and $p$ are any three positive integers. Then%
\begin{equation*}
\underset{r\rightarrow +\infty }{\underline{\lim }}\frac{\log ^{\left[ m-1%
\right] }T_{h}^{-1}T_{f}\left( r\right) }{\log ^{\left[ n-1\right]
}T_{k}^{-1}T_{g}\left( r\right) }\leq \frac{\overline{\tau }_{h}^{\left(
m,q\right) }\left( f\right) }{\sigma _{k}^{\left( n,q\right) }\left(
g\right) }\leq \underset{r\rightarrow +\infty }{\overline{\lim }}\frac{\log
^{\left[ m-1\right] }T_{h}^{-1}T_{f}\left( r\right) }{\log ^{\left[ n-1%
\right] }T_{k}^{-1}T_{g}\left( r\right) }~.
\end{equation*}
\end{theorem}

\begin{theorem}
\label{t18} Let $f$, $g$ be any two meromorphic functions and $h,$ $k$ be
any two entire functions such that $0<$ $\tau _{h}^{\left( m,q\right)
}\left( f\right) $ $\leq $ $\overline{\tau }_{h}^{\left( m,q\right) }\left(
f\right) $ $<\infty $, $0<$ $\overline{\sigma }_{k}^{\left( n,q\right)
}\left( g\right) $ $\leq $ $\sigma _{k}^{\left( n,q\right) }\left( g\right) $
$<\infty $ and $\lambda _{h}^{\left( m,q\right) }\left( f\right) $ $=$ $\rho
_{k}^{\left( n,q\right) }\left( g\right) $ where $m,n$ and $p$ are any three
positive integers. Then%
\begin{multline*}
\underset{r\rightarrow +\infty }{\underline{\lim }}\frac{\log ^{\left[ m-1%
\right] }T_{h}^{-1}T_{f}\left( r\right) }{\log ^{\left[ n-1\right]
}T_{k}^{-1}T_{g}\left( r\right) }\leq \min \left\{ \frac{\tau _{h}^{\left(
m,q\right) }\left( f\right) }{\overline{\sigma }_{k}^{\left( n,q\right)
}\left( g\right) },\frac{\overline{\tau }_{h}^{\left( m,q\right) }\left(
f\right) }{\sigma _{k}^{\left( n,q\right) }\left( g\right) }\right\} \\
\leq \max \left\{ \frac{\tau _{h}^{\left( m,q\right) }\left( f\right) }{%
\overline{\sigma }_{k}^{\left( n,q\right) }\left( g\right) },\frac{\overline{%
\tau }_{h}^{\left( m,q\right) }\left( f\right) }{\sigma _{k}^{\left(
n,q\right) }\left( g\right) }\right\} \leq \underset{r\rightarrow +\infty }{%
\overline{\lim }}\frac{\log ^{\left[ m-1\right] }T_{h}^{-1}T_{f}\left(
r\right) }{\log ^{\left[ n-1\right] }T_{k}^{-1}T_{g}\left( r\right) }~.
\end{multline*}
\end{theorem}

\begin{remark}
The same results of above theorems in terms of Maximum modulus of entire
functions can also be deduced if we consider $f$ and $g$ be any two entire
functions.
\end{remark}

\end{document}